\documentclass[12pt]{amsart}
\usepackage{amsmath,amssymb,amsbsy,amsfonts,latexsym,amsopn,amstext,cite,
                                              amsxtra,euscript,amscd,bm}
\usepackage{url}
\usepackage[colorlinks,linkcolor=blue,anchorcolor=blue,citecolor=blue,backref=page]{hyperref}
\usepackage{color}
\usepackage{graphics,epsfig}
\usepackage{graphicx}
\usepackage{float}

\hypersetup{breaklinks=true} 

\RequirePackage{mathrsfs} \let\mathcal\mathscr

\usepackage[norefs,nocites]{refcheck} 

\usepackage{mathtools}
\usepackage{todonotes}
\usepackage{url}

\begin{document}
\newtheorem{theorem}{Theorem}
\newtheorem{lemma}[theorem]{Lemma}
\newtheorem{claim}[theorem]{Claim}
\newtheorem{cor}[theorem]{Corollary}
\newtheorem{prop}[theorem]{Proposition}
\newtheorem{definition}[theorem]{Definition}
\newtheorem{question}[theorem]{Question}
\newtheorem{rem}[theorem]{Remark}
\newcommand{\hh}{{{\mathrm h}}}

\numberwithin{equation}{section}
\numberwithin{theorem}{section}
\numberwithin{table}{section}

\def\sssum{\mathop{\sum\!\sum\!\sum}}
\def\ssum{\mathop{\sum\ldots \sum}}
\def\dsum{\mathop{\sum\sum}}
\def\iint{\mathop{\int\ldots \int}}

\def \vol{\mathrm{vol}\, }

\def\squareforqed{\hbox{\rlap{$\sqcap$}$\sqcup$}}
\def\qed{\ifmmode\squareforqed\else{\unskip\nobreak\hfil
\penalty50\hskip1em\nobreak\hfil\squareforqed
\parfillskip=0pt\finalhyphendemerits=0\endgraf}\fi}

\newfont{\teneufm}{eufm10}
\newfont{\seveneufm}{eufm7}
\newfont{\fiveeufm}{eufm5}
%
%
\newfam\eufmfam
    \textfont\eufmfam=\teneufm
\scriptfont\eufmfam=\seveneufm
    \scriptscriptfont\eufmfam=\fiveeufm
%
%
\def\frak#1{{\fam\eufmfam\relax#1}}

\newcommand{\bflambda}{{\boldsymbol{\lambda}}}
\newcommand{\bfmu}{{\boldsymbol{\mu}}}
\newcommand{\bfxi}{{\boldsymbol{\xi}}}
\newcommand{\bfrho}{{\boldsymbol{\rho}}}

\def\fK{\mathfrak K}
\def\fT{\mathfrak{T}}

\def\fA{{\mathfrak A}}
\def\fB{{\mathfrak B}}
\def\fC{{\mathfrak C}}
\def\fM{{\mathfrak M}}

\def \balpha{\bm{\alpha}}
\def \bbeta{\bm{\beta}}
\def \bgamma{\bm{\gamma}}
\def \blambda{\bm{\lambda}}
\def \bchi{\bm{\chi}}
\def \bphi{\bm{\varphi}}
\def \bpsi{\bm{\psi}}
\def \bomega{\bm{\omega}}
\def \btheta{\bm{\vartheta}}
\def \bnu{\bm{\nu}}
\def \beps{\bm{\varepsilon}}
\def \bxi{\bm{\xi}}

\def\eqref#1{(\ref{#1})}

\def\vec#1{\mathbf{#1}}
\renewcommand{\d}{\mathrm{d}}
\renewcommand{\phi}{\varphi}
\newcommand{\1}{\mathbf{1}}


\def\cA{{\mathcal A}}
\def\cB{{\mathcal B}}
\def\cC{{\mathcal C}}
\def\cD{{\mathcal D}}
\def\cE{{\mathcal E}}
\def\cF{{\mathcal F}}
\def\cG{{\mathcal G}}
\def\cH{{\mathcal H}}
\def\cI{{\mathcal I}}
\def\cJ{{\mathcal J}}
\def\cK{{\mathcal K}}
\def\cL{{\mathcal L}}
\def\cM{{\mathcal M}}
\def\cN{{\mathcal N}}
\def\cO{{\mathcal O}}
\def\cP{{\mathcal P}}
\def\cQ{{\mathcal Q}}
\def\cR{{\mathcal R}}
\def\cS{{\mathcal S}}
\def\cT{{\mathcal T}}
\def\cU{{\mathcal U}}
\def\cV{{\mathcal V}}
\def\cW{{\mathcal W}}
\def\cX{{\mathcal X}}
\def\cY{{\mathcal Y}}
\def\cZ{{\mathcal Z}}
\newcommand{\rmod}[1]{\: \mbox{mod} \: #1}

\def\sA{{\mathscr A}}
\def\sL{{\mathscr L}}
\def\sM{{\mathscr M}}
\def\sN{{\mathscr N}}
\def\sE{{\mathscr E}}
\def\tsL{\widetilde \sL}

\def\cg{{\mathcal g}}

\def\vr{\mathbf r}

\def\e{{\mathbf{\,e}}}
\def\ep{{\mathbf{\,e}}_p}
\def\eq{{\mathbf{\,e}}_q}

\newcommand{\Discr}{\operatorname{Discr}}
\newcommand{\Span}{\operatorname{Span}}
\newcommand{\End}{\operatorname{End}}
\newcommand{\Mat}{\operatorname{Mat}}

\newcommand{\sign}{\operatorname{sign}}
\newcommand{\rank}{\operatorname{rank}}
\newcommand{\Conv}{\operatorname{Conv}}

\newcommand{\RE}{\operatorname{\mathfrak{Re}}}
\newcommand{\IM}{\operatorname{\mathfrak{Im}}}

\newcommand{\pwgcd}{\operatorname{pw-gcd}}

\def\Tr{{\mathrm{Tr}}}
\def\Nm{{\mathrm{Nm}}}
\def\GL{{\mathrm{GL}}}

\def\SS{{\mathbf{S}}}

\def\lcm{{\mathrm{lcm}}}

\newcommand{\ba}{{\bf a}} 

\def\({\left(}
\def\){\right)}
\def\l|{\left|}
\def\r|{\right|}
\def\fl#1{\left\lfloor#1\right\rfloor}
\def\rf#1{\left\lceil#1\right\rceil}

\def\mand{\qquad \mbox{and} \qquad}




\hyphenation{re-pub-lished}

\mathsurround=1pt

\def\bfdefault{b}

\def \F{{\mathbb F}}
\def \K{{\mathbb K}}
\def \N{{\mathbb N}}
\def \Sq{\N^{(2)}}
\def \Nk{\N^{(k)}}
\def \NknH{N_n^{(k)}(H)}%
\def \sNknH{\sN_n^{(k)}(H)}%
\def \LknH{L_n^{(k)}(H)}%
\def \sLknH{\sL_n^{(k)}(H)}%
\def \sLkjH{\sL_j^{(k)}(H)}%
\def \bLknH{\overline{L}_n^{(k)}(H)}%
\def \bLkjH{\overline{L}_j^{(k)}(H)}%
\def \bsLknH{\overline{\sL}_n^{(k)}(H)}%
\def \bsLkjH{\overline{\sL}_j^{(k)}(H)}%
\def \Z{{\mathbb Z}}
\def \Q{{\mathbb Q}}
\def \R{{\mathbb R}}
\def \C{{\mathbb C}}
\def\Fp{\F_p}
\def \fp{\Fp^*}
\def\Nsq{\N^\sharp}

\newcommand{\gal}{\operatorname{Gal}}
\newcommand{\Sqf}{\cS}  
\newcommand{\ftwo}{{\mathbb F}_2} 
\newcommand{\fq}{{\mathbb F}_q}
\newcommand{\fqn}{{\mathbb F}_{q^n}}
\newcommand{\fixme}[1]{\footnote{Fixme: #1}}

\def \xbar{\overline x}
 \def \ybar{\overline y}

\title[Perfect power products  and multiquadratic extensions]{On the number of products which form perfect powers and discriminants
of multiquadratic extensions}

\author[R. de la Bret\`eche] {R\'egis de la Bret\`eche}
\address{ Institut de Math\'ematiques de Jussieu, UMR 7586, 
Universit\'e Paris-Diderot, 
UFR de Math\'ematiques, case 7012, 
B\^atiment Sophie Germain, 
75205 Paris Cedex 13, France} 
\email{regis.de-la-breteche@imj-prg.fr}

\author[P. Kurlberg]{P\"ar Kurlberg}
\address{Department of Mathematics, 
Royal Institute of Technology\\ SE-100 44 Stockholm, Sweden}
\email{kurlberg@math.kth.se}

\author[I. E. Shparlinski] {Igor E. Shparlinski} 
\address{Department of Pure Mathematics, University of New South Wales,
Sydney, NSW 2052, Australia}
\email{igor.shparlinski@unsw.edu.au}

\begin{abstract} We study some counting questions concerning products of 
positive integers $u_1, \ldots, u_n$
which form a nonzero perfect square, or more generally, a perfect
$k$-th power. We obtain an asymptotic formula for the 
number of such integers of bounded size and in particular  improve and generalize
a result  of D.~I.~Tolev~(2011).
We also use similar ideas to count the discriminants of number fields which are
multiquadratic extensions of~$\Q$ and improve and generalize a result  of  N.~Rome (2017).
\end{abstract}


\date{\today}
\maketitle
\section{Introduction} 

\subsection{Background and motivation}

Here we use a unified approach to study two intrinsically related problems:
\begin{itemize}
\item we count the number of integer vectors which are multiplicatively dependent 
modulo squares or higher powers, in particular we improve a result of  Tolev~\cite{Tol};  
\item we obtains some statistics for towers of radical extensions and extend and improve 
results of   Baily~\cite{Bail} and Rome~\cite{Rome}. 
\end{itemize}
Our  treatment  of  both problems  is based on similar ideas, namely, on  multiplicative decompositions 
close to those used in~\cite{B17}, see~\eqref{eq:mult_decomp1} and~\eqref{eq:mult_decomp2} in 
the proofs of  Theorems~\ref{thm:NHk} and~\ref{thm:FnH}, 
respectively,  which are our main results. 

More precisely, we study the following two groups of questions.

For a fixed integer $n \geqslant 2$ we are, in particular,  interested in the distribution 
on $n$-dimensional vectors of positive integers 
$$
\vec{a} = (a_1,\ldots,a_n)\in \N^n
$$ 
whose nontrivial sub-product $a_{i_1} \ldots a_{i_m}$, $ 1 \leqslant i_1 < \ldots < i_m \leqslant n$, 
is a perfect square. This seems to be a natural analogue of the question of counting 
multiplicatively dependent vectors~\cite{PSSS}.

Motivated by applications to integer factorisation algorithms a question of the existence 
of such a perfect square   amongst $n$ randomly selected integers of size at most $H$, 
has been extensively studied, 
see~\cite{CGPT,Pom1,Pom2}. 
More precisely, for the above applications it is crucial to determine the smallest value of
$n$ (as a function of $H$) for which at least one  such products is a
perfect square 
with a 
probability close to one; this question  has recently been answered in a spectacular work 
of Croot,  Granville,  Pemantle and Tetali~\cite{CGPT}. 

Further motivation for this work comes from studying the
multiquadratic extensions of $\Q$, that is, fields of the form
\begin{equation}
\label{eq:Quadr Tow}
\Q\(\sqrt\vec{a}\) = \Q\(\sqrt{a_1},\ldots,\sqrt{a_n}\)
\end{equation}
with vectors $\vec{a} = (a_1,\ldots,a_n)\in \N^n$ (or in $\Z^n$),  see, for
example,~\cite{Bail,BLT,Rome} and references therein.  
In particular we count the number of distinct discriminants of such fields up 
a certain bound $X$, and we also count the number of vectors 
$\vec{a}$ in a box for which $\Q\(\sqrt\vec{a}\) $ has the largest possible Galois group
$\gal(\Q(\sqrt{\vec{a}})/\Q) \simeq (\Z/2\Z)^{n}$. 
Finally, we also consider towers of radical extensions of higher degree $k \geqslant 2$ and
count the number of vectors 
$\vec{a}$ in a box for which these extensions are of the largest possible
degree $k^n$.

\subsection{Our results} Our main focus is on products forming squares
when $n$ is fixed, and thus it is easy to see that the existence of a
square product is a rare event. Furthermore, in this case, one can
concentrate on the case when such products include all numbers
$u_1,\ldots,u_n$.

In particular, we are interested in counting such vectors and more generally, vectors for which 
$u_1\ldots u_n$ is a perfect $k$-th power, for a fixed integer $k\geqslant 2$
in the hypercube 
\begin{equation}
\label{eq:Box}
\fB_n\(H\)  = [ 1, H]^n, 
\end{equation}
where $H \in \N$. 
In particular, we study  the cardinality
$$
\NknH = \#\sNknH
$$
 of the set 
$$\sNknH=\left\{(u_1,\ldots,u_n)\in \N^n \cap \fB_n\(H\):~u_1\cdots u_n  \in \Nk
\right\}, 
$$ 
where 
$$
\Nk = \{s^k:~s \in \N\} 
$$
denotes the set of positive integers which are perfect  $k$-th powers.

We note that if $\tau_{n,H}(s)$  denotes the restricted $n$-ary divisor function of $s\in \N$,
that is the number of representation $u_1\ldots u_n = s$ with integers $1 \leqslant u_1,\ldots,u_n \leqslant H$
then 
$$
\NknH= \sum_{s \leqslant H^{n/k}} \tau_{n,H}(s^k).
$$

Here we obtain an asymptotic formula for $N^{(k)}_n(H)$ and then
make it more explicit in the case of squares, that is for $k=2$. In turn this can be
used to study multiquadratic extensions of $\Q$ as in~\eqref{eq:Quadr Tow}. 

In particular, a combination of our results with  a result of
Balasubramanian, Luca and  
Thangadurai~\cite[Theorem~1.1]{BLT} allows to get an asymptotic
formula for the number of  
vectors  $\vec{a}\in \N^n \cap \fB_n\(H\)$ where $\fB_n\(H\)$ is given 
by~\eqref{eq:Box} for which
\begin{equation}
\label{eq:MaxDeg}
[\Q\(\sqrt\vec{a}\):\Q] =2^n.
\end{equation}
We also consider the more difficult questions of counting 
the discriminants  of multiquadratic number fields 

We recall that Rome~\cite{Rome}, making the result of
Baily~\cite[Theorem~8]{Bail} more precise, has recently given the
asymptotic formula for the number of distinct discriminants of size at
most $X$ coming from biquadratic fields
$\Q\big(\sqrt{a}, \sqrt{b}\big)$, see also~\cite[Section~6.1]{CDyDO}.
We also refer to~\cite{Bha,CDyDO,Klu,Won} for other counting result
for discriminants of quartic fields of different types. More
generally, using class field theory, Wright~\cite{W89}, extending
previous results of M{\"a}ki~\cite{Maki} on counting abelian
extensions of $\Q$, has obtained asymptotic formulas for counting
abelian extensions of global fields, though without giving explicit
leading constants and error terms.  We note that M{\"a}ki~\cite{Maki}
gives some (but not full) information about the main term and also obtains a
power saving in the error terms, see, for example~\cite[Theorems~10.5
and~10.6]{Maki}, which however are weaker than our result.
Here we obtain a generalisation of results of Baily~\cite{Bail} and
Rome~\cite{Rome} to multiquadratic extensions $\Q\(\sqrt\vec{a}\)$
for arbitrary length $n \geqslant 2$.

Furthermore, we also count distinct multiquadratic fields having
maximal Galois group, as well as the analogous question regarding
maximal degree extensions generated by higher {\em odd} index radicals
(that is, extension of the form
$ \Q\(\sqrt[k]{\vec{a}}\) = \Q\(\sqrt[k]{a_1},\ldots,\sqrt[k]{a_n}\) $
for odd $k>2$; here $\sqrt[k]{a_i}$ can denote any $k$-th root
of $a_{i}$ but it is convenient to always take a real $k$-th root.)

Our method can easily be adjusted to count  $\vec{a}\in \Z^n \cap \fB_n^\pm\(H\)$ where 
$$
\fB_n^\pm\(H\)   = \([ -H, -1] \cup  [ 1, H]\)^n.
$$

\subsection{Notation} 
We recall
that the notations $U = O(V)$,  $U \ll V$ and  $V \gg U$  are
all equivalent to the statement that $|U| \leqslant c V$ holds
with some constant $c> 0$, which throughout this work may depend 
on the integer parameters $k, n \geqslant 1$, and occasionally, where obvious, 
on the real parameter $\varepsilon> 0$. 

We also   denote
$$
\N_0 = \N \cup \{0\}\mand  \R_+=\R\cap [0,\infty),
$$
and it is convenient to define
$$
\Z_* = \Z\smallsetminus\{0\}.
$$

Throughout the paper, the letter $p$ always denotes a prime number.

\section{Products  which form  powers}

\subsection{Products which are $k$-th powers}
We obtain an asymptotic formula, with a power saving  in the error term, for $\NknH$ for any 
integer $k \geqslant 2$ which  generalizes  and improves a result of Tolev~\cite{Tol} that corresponds to $n=2$
and gives only a logarithmic saving. 
We always write $\vec{m} = (m_1,\ldots,m_n)$ and  introduce the sets
\begin{align*}
&\sM_{n,k} =\{ \vec{m}\in \N_0^n\smallsetminus\{\vec{0}\}:~ k\mid m_1+\ldots+m_n\}, \\
& \sM_{n,k,i} =\{ \vec{m}\in \sM_{n,k}:~ m_1+\ldots+m_n=ik\},\\
& \sM^*_{n,k} =\{ \vec{m}\in \sM_{n,k,1}:~ \#\{ i\, : m_i>0\}\geqslant  2\},\\
&\cE_{n,k,i}=\big\{ \vec{\varepsilon}\in \{0,\ldots,k-1\}^n:~ \varepsilon_1+\ldots+\varepsilon_n=ki\big\}.
\end{align*} 
In particular, the set $\sM_{n,k,1}\smallsetminus \sM^*_{n,k}$ consists of the
$n$ vectors $\vec{m}$ with exactly one nonzero coordinate  which
equals $k$.  
We also
denote
\begin{equation}
\begin{split}
\label{eq:qnk}
 &q_{n,k}=\#\sM_{n,k,1} = \binom{n+k-1}{k}, \\
 &q_{n,k}^*=\# \sM^*_{n,k}=\# \sE_{n,k,1} =  q_{n,k} - n =    \binom{n+k-1}{k} - n. 
\end{split}
\end{equation}  
We consider the vectors  $\vec{t} \in \R_+^{q_{n,k}^*}$, with
components indexed by elements of  $\sM^*_{n,k} $, and define $I_{n,k}$ as the volume of the following polyhedron:
\begin{equation}
\begin{split}
\label{defIk}
I_{n,k}= \vol
\biggl\{\vec{t} = \(t_{\vec{m}}\)_{\vec{m} \in \sM^*_{n,k} }& \in \R_+^{q_{n,k}^*}~:\\
&~ \sum_{\vec{m}\in \sM^*_{n,k} } m_j t_{\vec{m}}  \leqslant 1, \ 
1\leqslant j\leqslant n \biggr\} .
\end{split}
\end{equation}

\begin{rem} Clearly the cube $[0,1/k]^{q_{n,k}^*}$ is inside of the region whose volume is 
measured by $ I_{n,k}$, Hence,  we have 
$$
k^{-q_{n,k}^*} \leqslant I_{n,k} \leqslant 1.
$$ 
\end{rem}

Using the results of~\cite{B01}, which we summarize in Section~\ref{sec:ArithFunc}, 
we derive the following asymptotic formula for $\NknH$.

\begin{theorem}
\label{thm:NHk} 
Let $n\geqslant  1$ and $k\geqslant  2$ be fixed.
There exists $\vartheta_{n,k}>0$ and $Q_{n,k}\in \R[X]$ of degree $q_{n,k}^*$, given by~\eqref{eq:qnk}, 
such that for any $H \geqslant 2$ we have
$$\NknH=H^{n/k}Q_{n,k}(\log H)+O(H^{n/k-\vartheta_{n,k}}),$$
where  the leading coefficient $C_{n,k}$ of $Q_{n,k}$ satisfies  
$$C_{n,k}= I_{n,k}\prod_p\(1-\frac1p\)^{q_{n,k}^* }\(1+\sum_{i=1}^\infty \frac{\#\sE_{n,k,i}}{p^i}\),
$$
where the product is taken over all prime numbers and  $I_{n,k}$ is defined in~\eqref{defIk}. 
\end{theorem}

\subsection{Products which are squares}

We now give more explicit form of Theorem~\ref{thm:NHk}
when $k=2$; this  is important for applications.

In this case we simplify the notation by setting 
$$
N_n(H) = N_n^{(2)}(H), \quad I_n = I_{n,k}, \quad C_n = C_{n,2}, \quad q_n = q_{n,2} \quad q_n^* = q_{n,2}^*.
$$

We now have from~\eqref{eq:qnk}  
$$
q_n=\frac{n(n+1)}{2} \mand q_n^* = \frac{n(n-1)}{2}. 
$$
Observing that 
$$
\#\cE_{n,2,i}= \binom{n}{2i},
$$ we derive
$$C_n= I_n\prod_p \(1-\frac{1}{p}\)^{n(n-1)/2}
\(\frac12 \(1+\frac{1}{p^{1/2}  }\)^n+ \frac12  \(1-\frac{1}{p^{1/2}  }\)^n\) , 
$$
where the product is taken over all prime numbers.

Let $\cH$ be the set of integers $h  \in [0, 2^n-1]$  with exactly two nonzero binary digits.
In particular, 
the first element of $\cH$ is $2+1=3$ and the largest element is $2^{n-1} + 2^{n-2} = 3\cdot 2^{n-2}$.

Then we see that   $I_n$ can now be defined as the volume of the
following polyhedron:
$$
{I}_n= 
\vol \left\{\vec{t} \in \R_+^{\cH}:~ \sum_{ h\in \cH} {\varepsilon_j(h)t_h} \leqslant 1, \    1\leqslant j\leqslant n \right\} , 
$$
where $\varepsilon_j(h)$ denotes the $j$-th digit in the binary expansion of $h$.

\begin{rem}  For numerical calculations we can add another condition 
$t_3\leqslant \ldots \leqslant t_{3\cdot 2^{n-2}}$  and  then multiply
by   $  \(n(n-1)/2\)!$  the resulting integral. 
Thus, we have
$$I_2=1,\qquad I_3=6\int_{\substack{ 
0\leqslant t_3\leqslant t_5\leqslant t_6\leqslant 1-t_5}}\d \vec{t}=
6\int_0^{1/2} t_5(1-2t_5)\d t_5=\tfrac 14.$$
\end{rem}

We now see that for $k=2$, Theorem~\ref{thm:NHk} implies the following result.

\begin{cor}
\label{cor:NH} 
Let $n\geqslant  1$ be fixed.
There exists $\vartheta_n>0$ and $Q_n\in \R[X]$ of degree $ n(n-1)/2$ such that for any $H \geqslant 2$ we have
$$N_n(H)=
H^{n/2}Q_n(\log H)+O(H^{n/2-\vartheta_n}),
$$
where  the leading coefficient $C_n$ of $Q_n$ satisfies  
$$C_n= I_n\prod_p \(1-\frac{1}{p}\)^{n(n-1)/2}
\(\frac12 \(1+\frac{1}{p^{1/2}  }\)^n+ \frac12  \(1-\frac{1}{p^{1/2}  }\)^n\), 
$$
where the product is taken over all prime numbers. 
\end{cor}

In particular, for $n=2$, we have 
\begin{align*}
C_2 & = I_2 \prod_p \(1-\frac{1}{p}\)
\(\frac12 \(1+\frac{1}{p^{1/2}  }\)^2+ \frac12  \(1-\frac{1}{p^{1/2}  }\)^2\)\\
& = \prod_p \(1-\frac{1}{p}\) \(1+\frac{1}{p}\)
=  \prod_p \(1-\frac{1}{p^2}\) = \zeta(2)^{-1} =\frac{6}{\pi^2}, 
\end{align*}
where $\zeta$ is the Riemann zeta-function.

\section{Counting multiquadratic fields} 

\subsection{Discriminants  of multiquadratic fields}
\label{sec:Discr}

Let $F_n(X)$ be the number of distinct fields $\Q\(\sqrt\vec{a}\)$ with $\vec{a} \in \Z^n$  
of  largest possible degree  
as in~\eqref{eq:MaxDeg} whose  discriminant over $\Q$
satisfy
$$
 \Discr \Q\(\sqrt\vec{a}\)  \leqslant X.
$$

Let us define
 \begin{equation}
\label{eq: def hn} 
t_n= \prod_{k=0}^{n-1} (2^n-2^k).
\end{equation}

\begin{theorem}
\label{thm:FnH} 
Let $n\geqslant  1$ and $\varepsilon > 0$ be fixed. There exists a
polynomial $P_n$  of degree $2^n-2$ with the
leading coefficient  
\begin{align*}
A_n=\frac{4^n +5\cdot 2^n+10}{2^{3+(n-1)(2^n-2)}({2^n+1})(2^n-2)!t_n}& \\
\prod_p\(1-\frac1p\)^{2^n-1}&\(1+\frac{2^n-1}p\),
\end{align*}
such that, for $X\geqslant  2$ ,
$$ 
F_n(X)= X^{1/2^{n-1}} \( P_n(\log X)+O_{\varepsilon}\(X^{-\eta_n +\varepsilon}\)\), 
$$
where
$$
\eta_n = \frac{3}{2^{n-1}(5+2^n)}.
$$
\end{theorem}

We remark that
Rome~\cite{Rome} has obtained a special case of Theorem~\ref{thm:FnH}
 for $n=2$, however with a larger error term, see also~\cite{Bail, W89}. 
 A version of Theorem~\ref{thm:FnH} is also given by
 Fritsch~\cite{Frit}. His  method
 is more elementary and gives a weaker bound on error term, though also with
 a power saving.

 Let $f_n(d)$ be the number of distinct fields $\Q\(\sqrt\vec{a}\)$ with $\vec{a} \in \N^n$ of  largest possible degree  as in~\eqref{eq:MaxDeg} whose  discriminants over $\Q$ satisfy 
$\Discr Q\(\sqrt\vec{a}\)=d$.

We now explicitly evaluate the generating series
$$
g_n(s) = \sum_{d=1}^\infty \frac{f_n(d)}{d^{s}}, \qquad s \in \C.
$$
For this we define
\begin{equation}
\label{eq:Prod H}
h_n(s)=\prod_{p>2} \(1+\frac{2^n-1}{p^s}\),   \qquad s \in \C, \ \RE  s > 1. 
\end{equation}

\begin{theorem}
\label{thm:GenSer}
Let $n\geqslant  1$ be fixed. For any $s\in \C$ with $\RE s >1/2^{n-1}$ we have
$$
g_n(s)
=\frac{h_n\(2^{n-1}s\)}{t_n} \(1+
\frac{ 2^{n }-1}{2^{ 2^{n }s}}  +\frac{ 2^{n+1}-2}{2^{3\cdot2^{n }s}} +\frac{4^n -3 \cdot 2^n +2}{2^{2^{n+1}s}} \).
$$
\end{theorem}

\subsection{Multiquadratic fields with maximal Galois groups}
\label{sec:Gal}

We also wish to determine the number of distinct multiquadratic fields of
the form $\Q(\sqrt{\vec{a}})$ for $\vec{a}\in  \N^n \cap \fB_n\(H\)$, 
that have  maximal Galois group  
$$
\gal(\Q(\sqrt{\vec{a}})/\Q) \simeq (\Z/2\Z)^{n},
$$   
that is,
\begin{align*}
&G_n(H)\\
& \quad = \# \left\{\Q\(\sqrt{\vec{a}}\):~\vec{a}\in  \N^n \cap \fB_n\(H\)\
\text{and}\ \# \gal\(\Q\(\sqrt{\vec{a}}\)/\Q\) = 2^n \right \}.
\end{align*}

\begin{theorem}
\label{thm:MaxGal} 
We have, as $H \to \infty$,
$$
G_n(H) 
= \(\frac{1}{n!\zeta(2)^n}+O\({\rm e}^{-(1+o(1))\sqrt{ (\log H)(\log \log H)/2}}\)\)   H^{n}  .
$$
\end{theorem}

\subsection{Higher index radical extensions with maximal degree.}
\label{sec:gg}

Let $k \geqslant  3$ be an odd integer.    We can also determine the number of distinct fields
$$
K_{\ba} = \Q(\sqrt[k]{\vec{a}}) = \Q(\sqrt[k]{a_{1}}, \ldots,
\sqrt[k]{a_{k}}),
$$ 
where $\sqrt[k]{a_{i}}$ always denotes the real
$k$-th root of $a_{i}$, for $\vec{a}\in \N^n \cap \fB_n\(H\)$, that
have maximal degree,  that is
$$
G_n^{k}(H)
= \# \left\{\Q\(\sqrt[k]{\vec{a}}\):~\vec{a}\in  \N^n \cap \fB_n\(H\)\
\text{and}\  [\Q(\sqrt[k]{\vec{a}}):\Q] = k^{n} \right \}.
$$
Clearly $K_{\ba}$ is never Galois since
$K_{\ba} \subseteq \R$ and the Galois closure of $K_{\ba}$ must contain
the $k$-th cyclotomic extension $Z_{k} = \Q(\zeta_k)$, where
$\zeta_k$ is some fixed primitive $k$-th root of unity.

\begin{theorem}
 \label{thm:MaxDeg}
Let $k \geqslant  3$ be an odd integer.
Then, as $H \to \infty$,
$$
G_n^{k}(H) 
= \(\frac{1}{n!\zeta(k)^n}+O\({\rm e}^{-(1+o(1))\sqrt{ (\log H)(\log \log H)/2}}\)\)  H^{n}  .
$$
\end{theorem}

We remark that the general case of adjoining any choice of $k$-th
roots (possibly complex) to $\Q$ follows easily from the  case of real roots. 
Namely, for extensions of maximal degree, Kummer theory, see, for example,~\cite[Section~14.7]{DuFo} 
or~\cite[Chapter~VI, Sections~8--9]{lang-algebra},  implies that
the absolute Galois group acts transitively on the set of $n$-tuples
of the form
$\( \zeta_k^{e_{1}} \sqrt[k]{a_{1}}, \ldots, \zeta_k^{e_{n}}
\sqrt[k]{a_{n}}\)$, 
as $e_{1}, \ldots, e_{n}$ ranges over integers in $[1,k]$.

Further, since $K_{\ba}(\zeta_{k})$ is the normal closure of
$K_{\ba}$, it follows from Kummer theory (cf.
Section~\ref{sec:proof-theorem-maxdeg}) that
$\gal(K_{\ba}(\zeta_{k})/\Q)$ is maximal if and only if
$[K_{\ba}:\Q] = k^{n}$.  In particular, Theorem~\ref{thm:MaxDeg}
also allows us to count fields $K_{\ba}$ such that the normal closure has
maximal Galois group.
In fact, it is not difficult to show that the number of 
$\vec{a} = (a_{1}, \ldots, a_{n})\in  \N^n \cap \fB_n\(H\)$ such that
$a_{1}, \ldots, a_{n}$ are multiplicatively dependent modulo $k$-th powers is
$o(H^{n})$,   so Theorem~\ref{thm:MaxDeg} easily yields 
an asymptotic formula for the number of distinct  fields $K_{\vec{a}}$, as well
as an asymptotic formula   for the number of distinct normal closures
$K_{\vec{a}}(\zeta_{k})$, as $\vec{a}$ ranges over elements in
$\N^n \cap \fB_n\(H\)$.

\section{Sums of arithmetical functions of several variables}
\label{sec:ArithFunc}

\subsection{Setup} 

We say that $f$ is a   multiplicative   function of $\N^m$
if 
\begin{equation}
\label{eq:mult func}
f(e_1,\ldots,e_m)f(d_1,\ldots,d_m)=f(e_1d_1,\ldots,e_m d_m)
\end{equation}  
for all pairs of tuples of positive integers with 
$$\gcd(e_1\cdots e_m, d_1\cdots d_m)=1.$$
We next recall  some results of 
La~Bret\`eche~\cite[Theorems~1 and~2]{B01}, which
for  a nonnegative
multiplicative function $f$, links the  sum 
\begin{equation}
\label{multiplesums} 
S_{\bbeta}(X)= \sum_{1\leqslant d_1\leqslant X^{\beta_1}}\ldots
\sum_{1\leqslant d_m\leqslant X^{\beta_n}} f(d_1,\ldots,d_m),
\end{equation} 
where $\bbeta = \(\beta_1,\ldots, \beta_m\) \in \R^m$,
to the behavior of the associated multiple Dirichlet series
$$F(s_1,\ldots,s_m)=\sum_{d_1=1}^{\infty}\ldots\sum_{d_m=1}^{\infty} \frac{f(d_1,\ldots,d_m)}{d_1^{s_1}\cdots d_m^{s_m}}.
$$  

The goal is to understand analytic properties of $F$ in order to
obtain a tauberian theorem for multiple Dirichlet series. This is for
instance possible when $F$ can be written as an Euler product. As in
the one dimensional case, this is equivalent to the multiplicativity
of $f$.

In that case, formally we have  
$$F(\vec{s})=\prod_{p~\mathrm{prime}} \(\sum_{\bnu \in \N_0^{m}} \frac{f(p^{\nu_1},\ldots,p^{\nu_m})}{p^{\nu_1s_1+\cdots+\nu_ms_m}}\),
$$ 
where  $\bnu = \(\nu_1,\ldots,\nu_m\)$. 

To state the relevant results from~\cite{B01}  we need further  notations. We denote by $\cL_{m}(\C)$ the space of linear forms
$$
\ell(X_1, \ldots, X_{m})  \in \C[X_1, \ldots, X_{m}].
$$   
Let $\left\{\vec{e}_j\right\}_{j=1}^{m}$  be the canonical 
basis of $\C^{m}$ and let be $\left\{\vec{e}_j^*\right\}_{j=1}^{m}$ 
the dual basis in $\cL_{m}(\C)$. We denote by $\cL\R_{m}(\C)$   the set of linear forms of $\cL_{m}(\C)$  such that their restriction to $\R^{m}$ maps to $\R$.
We define $\cL\R_{m}^{+}(\C)$ similarly with respect to the set  $\R_+$ of 
nonnegative 
real numbers. 

As usual, we use  $\|\cdot\|_1$ to denote the $L^1$-norm  and  use 
$\langle \cdot \rangle$ to denote the inner product of vectors from $\R^{m}$. 

We view $\R^m$ as a partially ordered set using the relation 
$\vec{d}>\vec{e}$ if and only if this inequality holds component-wise 
for $\vec{d}, \vec{e} \in  \R^{m}$.

We also apply the notations $\RE$ and $\IM$, for the real and imaginary part,
to vectors in the natural component-wise fashion. 

\subsection{Asymptotic formula} 

We are now able to state~\cite[Theorem~1]{B01} which gives an asymptotic formula
for the sums $S_{\bbeta}(X)$ 
given by~\eqref{multiplesums}.

\begin{lemma}\label{DLB1} Let $f$ be a
nonnegative 
 arithmetical function on $\mathbb{N}^{m}$ and $F$ be the associated Dirichlet series
$$F(\vec{s})=\sum_{d_1=1}^{+\infty}\cdots\sum_{d_m=1}^{+\infty} \frac{f(d_1,\ldots,d_m)}{d_1^{s_1}\cdots d_m^{s_m}}.$$ 
We assume that there exists $\balpha\in\R_{+}^{m}$ such that $F$
satisfies the following properties:  
\begin{itemize}
\item[(P1)] $F(\vec{s})$ is absolutely convergent for $\vec{s}$ such that $\RE  (\vec{s}) > \balpha$.

\item[(P2)]There exists a family of $N$ nonzero linear forms $\cL=\left\{\ell^{(i)}\right\}_{i=1}^{N}$ of $\cL\R_{m}^{+}(\C)$ 
and a family of $R$ nonzero linear forms $\left\{h^{(r)}\right\}_{r=1}^R$ of $\cL\R_{m}^{+}(\C)$ and                                                                                                                            $\delta_{1},\delta_{3}>0$ such that the function $H$ from $\C^{m}$ to $\C$ defined by 
$$H(\vec{s}) = F(\vec{s}+\balpha)\prod_{i=1}^{N} \ell^{(i)}(\vec{s})$$
can be analytically continued in the domain
\begin{align*}
\cD(\delta_1,\delta_3)=\bigl\{\vec{s} \in\C^{m}:~\RE  \(\ell^{(i)}(\vec{s})\) &> -\delta_1,\ \forall i,  \ \text{and}\\
& \RE \(h^{(r)}(\vec{s})\)> -\delta_3 ,  \forall r
                                                                                                                          \bigr
                                                                                                                          \}
\end{align*}

\item[(P3)]There exists $\delta_2>0$ such that, for all $\varepsilon_1, \varepsilon_2>0$ the following upper bound
$$
H(\vec{s}) \ll \prod_{i=1}^{N}\(\vert\IM \(\ell^{(i)}(\vec{s})\)\vert + 1\)^{1-\delta_2\min\{0,\RE  \(\ell^{(i)}(\vec{s})\)\}}\(1+ \|\IM(\vec{s})\|_1^{\varepsilon_1}\)
$$
holds uniformly in the domain $\cD(\delta_1-\varepsilon_2, \delta_3-\varepsilon_2)$.  
\end{itemize}
Let $J(\balpha)=\{j\in \{1,\ldots,m\}:~\alpha_j=0\}$. 
We set $r=\#J(\balpha)$ and  let $\ell^{(N+1)},\ldots,\ell^{(N+r)}$ be the $r$ 
linear forms $\vec{e}_j^{*}$ where $j\in J(\balpha)$. 
Then, under previous hypotheses (P1), (P2) and (P3), there exists a polynomial $Q \in \R[X]$ of degree less or equal to 
$N+r-\rank(\{\ell^{(i)}\}_{i=1}^{N+r})$ and a real $\vartheta>0$, that depends on 
$\cL$, $\left\{h^{(r)}\right\}_{r=1}^R$, $\delta_1$, $\delta_2$, $\delta_3$, $\balpha$ 
and $\bbeta$, such that, for all $X\geqslant  1$,
we have
$$ S_{\bbeta}(X) 
=X^{\langle \balpha,\bbeta\rangle} \(Q(\log X) + O(X^{-\vartheta})\).
$$
\end{lemma}
 
We remark that in (P2) of  Lemma~\ref{DLB1} we have shifted the argument
of $F$ by $\balpha$ so that the critical point
is~$\vec{s} = \vec{0}$. 

Furthermore, the exact value of the degree of $Q$ is given by~\cite[Theorem~2]{B01}, which 
we  state in a form which is sufficient for our purpose. 
When $\cL=\left\{\ell^{(i)}\right\}_{i=1}^{n}$ is a finite subset of
$\cL\R_{m}^{+}(\C)$, we define  
$$
\Conv^{*}(\cL)=\sum_{\ell \in\cL} \R_{+}^{*} \ell.
$$
\goodbreak 

\begin{lemma}\label{DLB2} 
Let $f$ be an arithmetical function satisfying all the hypotheses of Lemma~\ref{DLB1}. Let $J(\balpha)=\{j\in \{1,\cdots,m\}: \alpha_j=0\}$. We set $r=\#J(\balpha)$ and $\ell^{(N+1)},\ldots,l^{(N+r)}$ the $r$ linear forms $\vec{e}_j^{*}$ where $j\in J(\balpha)$ as before. If $\rank(\{\ell^{(i)}\}_{i=1}^{N+r})=m$, $H(0,\ldots,0)\neq 0$ and 
$$
\sum_{j=1}^{m}\beta_j \vec{e}_j^* \in \Conv^{*}(\{\ell^{(i)}\}_{i=1}^{N+r}), 
$$ 
then $Q$ is a polynomial 
\begin{itemize}
\item of degree $D = N+r-m$, 
\item with the leading coefficient $H(0, \ldots, 0)  I$, 
where 
$$
I =\lim_{X\to+\infty} X^{-\langle \balpha,\bbeta\rangle}(\log X)^{-D}
\int_{\substack{\vec{y}\in [1,\infty)^{N}\\   
\prod_{i=1}^N y_i^{\ell_i(\vec{e}_j)}\leqslant X^{\beta_j}\\
1 \leqslant j \leqslant m}} \prod_{i=1}^N y_i^{\ell_i(\vec{\balpha})-1} \d \vec{y}. 
$$
with  $\vec{y} = (y_1, \ldots, y_N)$. 
\end{itemize}
\end{lemma}
\goodbreak

\section{Towers of quadratic extensions}

\subsection{Degree}
We now recall  a result of Balasubramanian,   Luca and
Thangadurai~\cite[Theorem~1.1]{BLT}  
which gives an explicit formula for the degrees of the 
fields~\eqref{eq:Quadr Tow}.

For  $\vec{a} = (a_1,\ldots,a_n)\in \Z_*^n$ we define the products
\begin{equation}
\label{eq: bJ}
b_\cJ=\prod_{j\in \cJ }a_j.
\end{equation}
Define  
$\gamma_\vec{a}$ as the number of subsets $\cJ\subseteq \{ 1,\ldots,n\}$
with 
$$
b_\cJ\in \Sq. 
$$
Note that since the empty set $\cJ$ is not excluded, 
we always have $\gamma_\vec{a}\geqslant 1$. 

Furthermore, we say that 
$\vec{a}$ is {\it multiplicatively  independent modulo   squares\/} if none 
of the products $b_\cJ$ with $\cJ\ne \varnothing$ is a square (that is, if  $\gamma_\vec{a} = 1$). 

\begin{lemma}\label{lem: deg}
 For  $\vec{a} = (a_1,\ldots,a_n)\in  \Z_*^n$ we have
$$
[\Q\(\sqrt\vec{a}\): \Q] = \frac{ 2^n} {\gamma _\vec{a}}. 
$$
\end{lemma}

Note  
 that $\gamma _\vec{a}$ is a power of $2$ as examining prime factorisation
of $a_1,\ldots,a_n$ we see that  this is the size of the kernel of some matrix over the field of two elements,
see also~\cite[Lemma~2.1]{BLT}. 
Hence the right hand side of the formula of Lemma~\ref{lem: deg} is
indeed an integer number.

\begin{cor}\label{cor:Max Deg}  For  $\vec{a} = (a_1,\ldots,a_n)\in  \Z_*^n$ the field 
$\Q\(\sqrt\vec{a}\)$ 
satisfies~\eqref{eq:MaxDeg}  if and only if $a_1,a_2,\ldots, a_n\in  \Z_*$ are  multiplicatively independent modulo  squares.
\end{cor}

Alternatively, since $\Q$ contains all roots of unity of order two,
Corollary~\ref{cor:Max Deg}  also follows from Kummer theory, cf.~\cite[Proposition~37, Chapter~14]{DuFo} 
or~\cite[Theorem~8.1, Chapter~VI, Section~8]{lang-algebra}.

\subsection{Discriminant}

First we recall that for a square-free  $a \in  \Z_*$ we have
\begin{equation}\label{discrQd}
\Discr \Q\(\sqrt{a}\) =   \begin{cases}
a, & \text{if}\ a \equiv 1 \pmod 4,\\
4a, & \text{if}\ a \equiv 2,3 \pmod 4.
\end{cases}
\end{equation}

We now examine the discriminant $\Discr \Q\(\sqrt\vec{a}\) $ of the
field $\Q\(\sqrt\vec{a}\)$  over $\Q$.  Since this is of independent interest and also  for future applications we establish a formula for  $\Discr\Q\(\sqrt\vec{a}\)$ which applies to $\vec{a} \in \Z^n$ rather than only for  $\vec{a} \in \N^n$.  

\begin{lemma}\label{lem: discr} Let $a_1,a_2,\ldots, a_n\in  \Z_*$ be  multiplicatively independent modulo squares.
Then 
$$\Discr\Q\(\sqrt\vec{a})\)=
\prod_{\substack{\cJ\subseteq \{ 1,\ldots,n\}\\ \cJ\neq  \varnothing}} \Discr\ \Q\(\sqrt{b_\cJ}\) > 0, 
$$ 
where the integers $b_\cJ$ are defined by~\eqref{eq: bJ}. 
\end{lemma}

\begin{proof} First we establish the positivity of $\Discr \Q\(\sqrt\vec{a}\)$ for 
 $n\geqslant 2$.     Indeed, if $\vec{a}\in \N^n$ then there is nothing to
prove. Otherwise we see that all embeddings of  
$\Q\(\sqrt\vec{a}\)$ are complex, and thus, recalling the multiplicative independence condition and 
Corollary~\ref{cor:Max Deg}, we see their number $r_2$ is given by  
$$
r_2  = \frac{1}{2}  [\Q\(\sqrt\vec{a}\): \Q]=2^{n-1}. 
$$
Since $n\geqslant 2$ we see that $r_2$ is even and by Brill's
theorem (see~\cite[Lemma~2.2]{Wash}), for the sign of the
discriminant,   we obtain 
$$
\sign \(\Discr\(\Q\(\sqrt\vec{a}\)\)\)= (-1)^{r_2} =1.
$$
  
Next, we show the product on the right hand side of the desired formula is also positive. 
Assume that the vector $\vec{a}$ has $k$ negative and $m$ positive components.
If $k= 0$ there is nothing to prove.   If $0 < k \le n$,
  we have exactly 
$2^{n-1}$ negative values  
among $b_\cJ$, $\cJ\subseteq \{ 1,\ldots,n\}$, and since $n \ge 2$ we
have the desired positivity again.

Hence  the desired equality is equivalent to 
$$
\left| \Discr\Q\(\sqrt\vec{a}\)\right| =
\prod_{\substack{\cJ\subseteq \{ 1,\ldots,n\}\\ \cJ\neq
\varnothing}} \left| \Discr \Q\(\sqrt{b_\cJ}\) \right|,
$$
which is a simple consequence of the {\it conductor-discriminant formula\/}
(see, for example,~\cite[Theorem~3.11]{Wash}).

Namely, given a Dirichlet character $\chi$, let $f_{\chi}$ denote its
conductor, and given a group $X$ of Dirichlet characters, let $K$ be the
number field associated with $X$.  Then the discriminant of $K$ is
given by
$$
\Discr K =
(-1)^{r_{2}} 
\prod_{\chi \in X } f_{\chi},
$$ 
where, as before, $r_2$  is the number of   are complex  embeddings. 

We apply this to $K = \Q(\sqrt{a_{1}}, \ldots, \sqrt{a_{n}})$,
under the assumption that $G=G(K/\Q) = (\Z/2\Z)^n$ and hence 
$X = \widehat{G}$ is the dual group.  We first note that any
nontrivial character $\chi \in \widehat{G}$ is quadratic, and its kernel
$\ker(\chi)$ can be identified with an index two subgroup of $G$.
Hence the fixed field $K^{\ker(\chi)}$ is a quadratic extension of
$\Q$, and any such character $\chi$ can be identifed with a Dirichlet
character associated with the quadratic extension $K^{\ker(\chi)}/\Q$.

Using the conductor-discriminant formula twice, we find that
$$
\left| \Discr  K^{\ker(\chi)}\right| = f_{\chi}, \qquad \forall \chi  \in \widehat{G}, 
$$
(note that $f_{\chi}=1$ if $\chi =\chi_0$ is trivial), as well as  
$$
\left| \Discr  K \right|  = \prod_{\chi  \in \widehat{G}} f_{\chi} =
\prod_{\chi   \in \widehat{G}\setminus \{\chi_0\}} \left|d(K^{\ker(\chi)})\right|.
$$
Now, $\{K^{\ker(\chi)}):~\chi   \in \widehat{G}\setminus \{\chi_0\} \}$ is exactly the set of
quadratic extension of $\Q$, contained in $K$, which in turn are
parametrised by  the elements of the set $ \{ \Q (\sqrt{b_\cJ} ):~\cJ\subseteq \{
  1,\ldots,n\},\  \cJ\neq    \varnothing\}$.  
\end{proof}

\subsection{Maximal Galois groups} 

Let $\ftwo$ denote the finite field with two elements.  Given $H \in \R_+$ we
consider an arbitrary $\ftwo$-vector space $V_{H}$, of dimension $\pi(H)$, 
where, as usual,  $\pi(H)$ denotes the number of primes $p \leqslant H$.

Let $\Sqf \subseteq \N$ denote the set of square-free positive
integers.  
Define a map
$\varphi_{H} :~(\Sqf \cap [1,H]) \to V_H$ by
\begin{equation}
\label{eq:Map phi}
\varphi_{H}(a) = (e_p \mod 2)_{p \leqslant H},
\end{equation}
where  
$$
a = \prod_{p \leqslant H} p^{e_{p}},
$$
and we identify $V_{H}$ with $\pi(H)$-tuples of elements in $\ftwo$,
indexed by primes $p \leqslant H$.

We now show that  $ \gal\(\Q\(\sqrt{\vec{a}}\)/\Q\)$ is maximal if and only if 
the vectors  $\varphi_{H}(a_1), \ldots, \varphi_{H}(a_n)$ are linearly independent
over $\ftwo$. 

\begin{lemma}
\label{lem:Span} 
 Given $\vec{a} \in (\Sqf \cap [1,H])^n$ we have $ \gal\(\Q\(\sqrt{\vec{a}}\)/\Q\) \simeq
 (\Z/2\Z)^{n}$ if and only if
$$
\dim_{\ftwo}\( \Span\( \varphi_{H}(a_1), \ldots, \varphi_{H}(a_n) \)\) =n.
$$
\end{lemma}

\begin{proof}
  The statement follows immediately from Kummer theory
  (cf.~\cite[Section~14.7]{DuFo}
  or~\cite[Chapter~VI, Sections~8--9]{lang-algebra}) since the
  relevant roots of unity, namely $\pm 1$, are in $\Q$.
\end{proof}

\section{Proofs of main results}

\subsection{Proof of Theorem~\ref{thm:NHk}}
As usual, for a prime $p$ and an integer $m \geqslant 0$  and $y\ne 0$, we use $p^{m}\parallel y$ 
to denote that 
$$
p^{m}\mid  y \mand p^{m+1}\nmid y.
$$
For $\vec{m}\in \sM_{n,k} $ and
$\vec{u}=(u_1, \ldots, u_{n}) \in \N^n$ we set
$$
u_{\vec{m}}=\prod_{\substack{p^{m_j}\parallel u_j\cr \forall
    j}}p
$$
(that is, a prime $p$ is included in the above product if  
and only if $p^{m_j}\parallel u_j$ for every $j=1, \ldots, n$, and
thus the product is finite since $\vec{m} \in \sM_{n,k}$ implies that
$m_j>0$ for at least one $j$).

Then we   parametrize the solutions of $u_1\cdots u_n=w^k$ as follows: 
\begin{equation}
\label{eq:mult_decomp1}
u_j=\prod_{\substack{\vec{m}\in \sM_{n,k}}} u_{\vec{m}}^{m_j},\qquad   1 \leqslant j \leqslant n.
\end{equation}

We note that this    parametrisation resembles   the one used in~\cite{B17}, yet it is different
in that no coprimality condition is imposed. 

We observe that 
$$N_n^{(k)}(H) = \#\left\{ (u_{\vec{m}})_{\vec{m}\in \sM_{n,k}}:~\prod_{\substack{\vec{m}\in \sM_{n,k}}} u_{\vec{m}}^{m_j}\leqslant H, \
 j=1, \ldots, n \right\},
 $$
where the vectors $(u_{\vec{m}})_{\vec{m}\in \sM_{n,k}}$
are formed from
all possible vectors $\vec{u}=(u_1, \ldots, u_{n}) \in \N^n$.

We now define $f(d_1,...,d_n)$ as the number of vectors  $(u_{\vec{m}})_{\vec{m}}$, 
 $\vec{m} = (m_1, \ldots, m_n) \in \sM_{n,k}$,  for which  we simultaneously have
$$
d_j = \prod_{\vec{m}\in \sM_{n,k}} u_{\vec{m}}^{m_j}, \qquad j = 1, \ldots, n.
$$
Clearly  $f(d_1, \ldots, d_n)$ is multiplicative as in~\eqref{eq:mult func}.

The multiple Dirichlet series associated to this counting problem is 
\begin{align*} 
F(\vec{s})&=\sum_{(u_{\vec{m}})_{\vec{m}\in \sM_{n,k}}  }\prod_{j=1}^n  \( \prod_{\substack{\vec{m}\in \sM_{n,k}}} u_{\vec{m}}^{m_j}\)^{-s_j} \\
&= \prod_p\(1+\sum_{ {\vec{m}} \in \sM_{n,k} }  \frac{1}{p^{ m_1s_1+\ldots+m_ns_n}}   \).
\end{align*}

Let   $\{ \ell_{ \vec{m}}\}_{ \vec{m}\in \sM_{n,k,1}}$ defined by
$$
\ell_{ \vec{m}}( \vec{s}) =\sum_{j=1}^n m_j s_j.
$$
There exists a  holomorphic function $G(\vec{s})$, which 
for any fixed $\varepsilon$ is uniformly bounded in the domain
$$
\left\{ \vec{s}\in \C^n :~ \RE \ell_{ \vec{m}}(s)\geqslant  \tfrac 1{2}+\varepsilon,\ \vec{m}\in \sM_{n,k,1} \right\}  
$$
such that 
$$
F(\vec{s}) = \prod_{ \vec{m}\in \sM_{n,k,1}}\zeta\(\ell_{ \vec{m}}( \vec{s}) \)    G(\vec{s}).
$$ 
To see this, note that this domain is in fact equal to
$$
\left\{ \vec{s}\in \C^n :~ \RE s_j\geqslant  \tfrac {1+2\varepsilon}{2k}, \ j=1,\ldots,n\right\}, 
$$
and for all ${\bf s}$ in this domain,
$G(\vec{s})$ is a product of terms of the form $P(\{p^{-\ell_{ \vec{m}}( \vec{s})}\}_{\vec{m}\in \sM_{n,k}})$ where $P$ is the polynomial defined by 
$$
P(\{X_{\vec{m}}\}_{\vec{m}\in \sM_{n,k}})=
\(1+\sum_{ {\vec{m}} \in \sM_{n,k} } X_{ \vec{m}}   \)  \prod_{ \vec{m}\in \sM_{n,k,1}}(1-X_{ \vec{m}}).
$$  
When one develops the product, the only monomial of degree $1$
corresponds to $\vec{m}\in \sM_{n,k,j}$ with $j\geqslant 2$. Further, for any
$j\geqslant 2$ and $\vec{m}\in \sM_{n,k,j}$, we have $\RE 
\ell_{ \vec{m}}({\bf s})\geqslant  1+2\varepsilon$ for all ${\bf s}$
in the domain,
and it is then  easy to
deduce the boundedness of $G(\vec{s})$.

We have
$$G\(\frac1k,\ldots,\frac1k\)=\prod_p\(1-\frac1p\)^{q_{n,k}}\(1+\sum_{i=1}^\infty \frac{\#\sM_{n,k,i}}{p^i}\).$$

We write $m_j=\varepsilon_j+kh_j$, where $\varepsilon_j\in\{ 0,k-1\}$ and 
$h_j\in\N_0$, $j =1, \ldots, n$. We have
$$ 1+\sum_{i=1}^\infty \frac{\#\sM_{n,k,i}}{p^i} 
= \(1-\frac1p\)^{ -n }\(1+\sum_{i=1}^\infty \frac{\#\cE_{n,k,i}}{p^i}\).
$$ 
We observe that   $k\mid m_1+\ldots +m_n$ is equivalent to  $ k\mid \varepsilon_1+\ldots +\varepsilon_n$. 
Then we have
$$G\(\frac1k,\ldots,\frac1k\)=\prod_p\(1-\frac1p\)^{q_{n,k}-n}\(1+\sum_{i=1}^\infty \frac{\#\sE_{n,k,i}}{p^i}\).$$

The Dirichlet series $F$ satisfies  the hypotheses of Lemma~\ref{DLB1} with 
$$\balpha =(\alpha_1, \ldots,\alpha_n) =\(\frac 1k,\ldots,\frac 1k\),
\ \text{ and } 
\bbeta =(\beta_1,\ldots,\beta_n) =(1,\ldots,1).
$$
One can check the hypothesis~P3 by using the bound 
$$\zeta(1+s)s\ll (1+|\IM s|)^{1-\RE  (s)/3+\varepsilon},\qquad \text{for } \RE s
\in \left[ -\tfrac 12, 0\right].
$$ 
which holds for any fixed $\varepsilon>0$.

Then there exists $\vartheta_{n,k}>0$, $Q_{n,k}\in \R[X]$  such that 
$$N_n^{(k)}(H)=H^{n/k}Q_{n,k}(\log H)+O(H^{n/k-\vartheta_{n,k}}).$$
We now apply Lemma~\ref{DLB2} with $N = \#\sM_{n,k,1}=q_{n,k}  $,
$$\{\ell^{(i)}\}_{1\leqslant i\leqslant N}=\{\ell_{\vec{m} }\}_{\vec{m}\in \sM_{n,k,1}}$$ 
and see that $ \deg Q_{n,k} = q_{n,k}^*$
since $\ell^{(j)}(\vec{s})=ks_j\in \{\ell_{\vec{m} }\}_{\vec{m}\in \sM_{n,k,1}}$
for all $1\leqslant j\leqslant n$. Then the set 
$\sM_{n,k}*$ is the subset of $\sM_{n,k,1}$ which  avoids the forms
$\{\ell^{(i)}\}_{1\leqslant i\leqslant N}$.
Moreover  
\begin{align*}
Q_{n,k}(\log H)&\sim \frac{G\(\tfrac 1k,\ldots,\tfrac 1k\)}{H^{n/k}}\int_{\substack{
(z_{\vec{m}} )\in 
[1,\infty)^{\#\sM_{n,k,1}}\\    \prod_{\vec{m}\in \sM_{n,k,1}} z_{\vec{m}}^{m_j}\leqslant H}}
\d  {z} 
\\&\sim G\(\tfrac 1k,\ldots,\tfrac 1k\)\int_{\substack{
(z_{\vec{m}} )\in 
[1,\infty)^{q_{n,k}}\\    \prod_{\vec{m}\in \sM^*_{n,k} } z_{\vec{m}}^{m_j}\leqslant H}}
\frac{\d \vec{z}}{\prod_{\vec{m}\in \sM^*_{n,k} } z_{\vec{m}}}
\\&\sim G\(\tfrac 1k,\ldots,\tfrac 1k\) I_{n,k}(\log H)^{ q_{n,k}^*}, 
\end{align*}
as $H\to \infty$, 
where $I_{n,k}$ is defined in~\eqref{defIk}. 
This  defines  the leading coefficient  of $Q_{n,k}$ and gives the desired result.

\subsection{Proof of Theorem~\ref{thm:FnH}} 
Let $K$ be a field counted by $F_n(X).$ 
There are $2^n-1$ quadratic extensions of $\Q$ in $K$. We write them as $\Q(\sqrt{c_j})$ with $1\leqslant j\leqslant 2^n-1$ where $c_j$ is square-free. 

We now recall that $t_n$ is defined by~\eqref{eq: def hn}.  Then, clearly, there are 
$t_n$  ways to choose $(j_1,\ldots, j_n)$ such that   $K= \Q\(\sqrt\vec{a}\)$ with the vector 
$\vec{a}=(c_{j_1},\ldots , c_{j_n}) \in \Z^n$.
The other $c_j$ can be calculated from $\vec{a}$ by choosing  for each of the remaining $j$ some unique set  $\cJ\subseteq\{ 1,\ldots, n\}$ 
of cardinality $\# \cJ \geqslant 2$ and calculating
$$\prod_{k\in\cJ}c_{j_k}=c_jd_j^2.$$
Then we have
\begin{align*} F_n(X)=\frac{1}{t_n}\#\bigl\{ &(a_1,\ldots,a_n)\in \Z^n:~ \mu^2(a_k)=1, \\
&\Discr\(\Q\(\sqrt\vec{a}\), \Q\)\leqslant X, \ [\Q\(\sqrt\vec{a}\): \Q]=2^n \bigr\}.
\end{align*}

Given square-free $a_1, \ldots, a_n \in \N$, we write
\begin{equation}
\label{eq:mult_decomp2}
a_j= \sigma_j 2^{\nu_j}\prod_{1\leqslant h\leqslant 2^n-1}z_h^{\varepsilon_j(h)}, \qquad j =1, \ldots, n, 
\end{equation}
where $\sigma_j \in \{ -1, 1\}$, $\nu_j  \in \{ 0,1\}$, $j=1, \ldots, n$, 
and $z_h$ are some odd positive integers., $h =1, \ldots,  2^n-1$.  

To see that the decomposition in ~\eqref{eq:mult_decomp2} is possible,
following~\cite{B17}, we number all nonempty subsets
$\cJ_h \subseteq\{ 1,\ldots, n\}$ and define $z_h$ as the greatest
common divisor of $a_j$, $j \in \cJ_h$.

 Since $a_1, \ldots, a_n$ are   square-free, the numbers $z_h$ are coprime.
For $\cJ\subseteq\{ 1,\ldots, n\}$, and $b_\cJ$ as in~\eqref{eq: bJ} we have 
\begin{equation}
\begin{split} 
\label{eq:defccJ}
b_\cJ   & = \prod_{j\in \cJ} a_j  =2^{n_\cJ}  s_\cJ  \prod_{j\in \cJ} \sigma_j  
\prod_{1\leqslant h\leqslant 2^n-1}z_h^{\sum_{j\in \cJ}\varepsilon_j(h)}\\
&  =2^{n_\cJ} s_\cJ  c_\cJ d_\cJ^2, 
\end{split} 
\end{equation}
where, as before, $\varepsilon_j(h)$ denotes the $j$-th digit in the binary
expansion of $h$,  
$$
n_\cJ=\sum_{j\in \cJ}\nu_j \mand  s_\cJ = \prod_{j\in \cJ} \sigma_j,
$$ 
and  $c_\cJ$ is odd and square-free.   
We have
$$c_\cJ=\prod_{\substack{1\leqslant h\leqslant 2^n-1\\ \sum_{j\in
      \cJ}\varepsilon_j(h)\equiv 1\pmod 2}}z_h.$$  
We write 
$$ \Discr \Q\(\sqrt{a_1},\ldots ,\sqrt{a_n}\) =2^W D, 
$$
where $D$ is odd.  

Using   Lemma~\ref{lem: discr} and the formula~\eqref{discrQd},
we derive from~\eqref{eq:defccJ} that 
$$D=\prod_{\substack{\cJ\subseteq \{ 1,\ldots,n\}\\ \cJ\neq \varnothing }}c_\cJ
=\prod_{\substack{1\leqslant h\leqslant 2^n-1 }}z_h^{\delta_h}$$
with
$$
\delta_h=2^{n-s(h)}\sum_{\substack{0\leqslant k\leqslant s(h)\\ k\equiv 1\pmod 2}}\binom{s(h)}{ k}=2^{n-1}, \qquad 1\leqslant h\leqslant 2^n-1,
$$
and where 
$$
s(h) = \sum_{j=1}^n \varepsilon_j(h)
$$ 
denotes the sum of digits in the binary expansion of
$h$.

Then $D$ is the largest odd divisor of 
$$
\lcm\(\vec{a}\)^{2^{n-1}} =\lcm\(a_1,\ldots,a_n\)^{2^{n-1}}.
$$

Let 
\begin{itemize}
\item $r_{1,4}(\cJ)$ be the number of $j\in \cJ$ such that $a_j\equiv 1\pmod 4$, 
\item  $r_{3,4}(\cJ)$ be the number of $j\in \cJ$ such that $a_j\equiv 3\pmod 4$, 
\item  $r_{2,8}(\cJ)$ be the number of $j\in \cJ$ such that $a_j\equiv 2\pmod 8$,
\item   $r_{6,8}(\cJ)$ be the number of $j\in \cJ$ such that $a_j\equiv 6\pmod 8$. 
\end{itemize}
We have 
$$r_{1,4}(\cJ)+r_{3,4}(\cJ)+r_{2,8}(\cJ)+r_{6,8}(\cJ)=\# \cJ.
$$
We now calculate $v_2\(\Discr ( \Q ( \sqrt{b_\cJ}),\Q )\)$, where $b_\cJ$ is as in~\eqref{eq: bJ}
and $v_2(m)$ denotes the largest power of $2$ dividing an integer $m\ne 0$.

Then we have 
\begin{align*}
v_2&\(\Discr \( \Q ( \sqrt{b_\cJ}),\Q \)\) \\
&\qquad  =   \begin{cases}
3, & \text{if}\ r_{2,8}(\cJ)+r_{6,8}(\cJ)\equiv 1\pmod 2,\\
2, & \text{if}\ r_{3,4}(\cJ) +r_{6,8}(\cJ)\equiv 1\pmod 2 ,\\
&\quad  \text{and}\ r_{2,8}(\cJ) +r_{6,8}(\cJ)\equiv 0\pmod 2,\\
0, & \text{otherwise.}
\end{cases}
\end{align*}
We now set   $\rho_{k_{1},k_{2}}=r_{k_{1},k_{2}}(\{ 1,\ldots, n\})$. We observe that 
$$\rho_{1,4}+\rho_{3,4}+ \rho_{2,8}+\rho_{6,8}=n.
$$

The number $U_3$  of $\cJ$ such that $v_2\(\Discr ( \Q ( \sqrt{b_\cJ}),\Q )\)=3$ is
$$ U_3 = \begin{cases} 2^{\rho_{1,4}+\rho_{3,4}+ \rho_{2,8}+\rho_{6,8}-1}=2^{n-1}&\text{if}\  \rho_{2,8}+\rho_{6,8}\geqslant  1,\\
0 & \text{if}\  \rho_{2,8}=\rho_{6,8}=0.
\end{cases}
$$ 

The number $U_2$ of $\cJ$ such that $v_2\(\Discr ( \Q ( \sqrt{b_\cJ}),\Q )\)=2$ is 
$$ 
U_2 = \begin{cases} 2^{\rho_{1,4}+\rho_{3,4}+ \rho_{2,8}+\rho_{6,8}-2}=2^{n-2}&\text{if}\   \rho_{3,4}+\rho_{6,8}\geqslant  1,  \\
&\quad   \rho_{2,8}+\rho_{6,8}\geqslant  1,\,   \rho_{2,8}+\rho_{3,4}\geqslant  1,\\
2^{ \rho_{3,4}-1}=2^{n-1}&\text{if}\  \rho_{3,4} \geqslant  1, \  \rho_{2,8}=\rho_{6,8}=0,\\
0&\text{if}\  \rho_{3,4} =\rho_{6,8}=0\\
& \quad \text{ or } \rho_{6,8}\geqslant  1,\ \rho_{2,8}+\rho_{3,4}=0. 
\end{cases}
$$

Using that 
$$
W = 3 U_3 + 2U_2
$$
we now deduce that   
$$
W= \begin{cases}
 2^{n+1} &\text{if}\   \rho_{3,4}+\rho_{6,8}\geqslant  1, \  \rho_{2,8}+\rho_{6,8}\geqslant  1,  \   \rho_{2,8}+\rho_{3,4}\geqslant  1,
\\
 3\cdot 2^{n-1} &\text{if}\  \rho_{3,4},\rho_{6,8}=0, \ \rho_{2,8} \geqslant  1,  \text{ or }  \rho_{3,4}, 
 \rho_{2,8} =0, \ \rho_{6,8} \geqslant  1,  \\
 2^{n } &\text{if}\   \ \rho_{3,4}\geqslant  1, \ \rho_{2,8}, \rho_{6,8}=0,\\
0&\text{if}\  \rho_{3,4},  \rho_{2,8} ,\rho_{6,8} =0  . 
\end{cases}
$$

Let $C_n(W)$ the number of possible configurations of the vectors
$\vec{a}$ corresponding to the four possibilities 
$$
1 \pmod 4, \qquad 3 \pmod 4, \qquad 2 \pmod 8, \qquad  6\pmod 8
$$
which correspond to a given  value  $W$.
Furthermore  when $\vec{z}$ and  a configuration is fixed the signs  $\sigma_1, \ldots, \sigma_n$ 
are also uniquely defined.  

In particular 
$$
\sum_{W\in
\{2^{n+1}, 3\cdot 2^{n-1}, 2^{n },0\}}C_n(W) = 4^n.
$$
More precisely, we have 
$$
C_n(W)= \begin{cases} 4^n -3 \cdot 2^n +2
  &\text{if}\  W=2^{n+1},
\\
 2^{n+1}-2 &\text{if}\  W=3\cdot 2^{n-1},   \\
 2^{n } -1&\text{if}\  W= 2^{n },\\
1&\text{if}\  W=0  . 
\end{cases}
$$

Let
$$T_n(  x)=  
\sum_{ \vec{z}\in \cZ}
\mu^2\(\prod_{1\leqslant h\leqslant 2^n-1} z_h\)   ,
$$
where
$$
\cZ = \{\vec{z}\in   \N^{2^{n}-1}:~ z_1,\ldots ,  z_{2^n-1}\ \text{odd and}\  z_1\ldots  z_{2^n-1}\leqslant  x\}.
$$
Then \begin{equation}\label{FnX} F_n(X)=\frac{1}{t_n} \sum_{W\in
\{2^{n+1}, 3\cdot 2^{n-1}, 2^{n },0\}}C_n(W)T_n\(\frac{X^{1/2^{n-1}}}{2^{W/2^{n-1}}}\).
\end{equation}

We have 
\begin{equation}
\label{eq:Tn mu}
T_n( x)= \sum_{\substack{m\leqslant x\\m~\text{odd}}}\mu^2(m) (2^n-1)^{\omega(m)}.
\end{equation}

By standard methods, there exists a
polynomial  $Q_{n} $ of degree
$2^n-2$ such that for 
$$
\kappa_n=3/(5+2^n)
$$
we have
$$ 
T_n(x)=\frac1{ (2^n-2)!}x\(Q_n (\log x)+O(x^{-\kappa_n+\varepsilon})\)
$$
for any $\varepsilon > 0$.
Moreover the leading coefficient of $Q_n $ is 
$$
B_n=\frac{2}{2^n+1}\prod_p\(1-\frac1p\)^{2^n-1}\(1+\frac{2^n-1}p\).
$$
Indeed, the associated Dirichlet series is $h_n(s)$ which is given 
by~\eqref{eq:Prod H}. 
It can be written as $h_n(s)=\zeta(s)^{2^n-1}\widetilde h_n(s)$ where 
$\widetilde h_n$ can be analytically continued until $\RE   s>\tfrac12$.  For more details, see~\cite[Exercise~194]{TW}.

From~\eqref{FnX}, we deduce that there exists  a
polynomial $P_n$ of
degree $2^n-2$ such that 
$$
F_n(X)= X^{1/2^{n-1}} \( P_n(\log X)+O\(X^{-\kappa_n/2^{n-1}+\varepsilon}\)\)
$$
for any $\varepsilon > 0$.
   Moreover the leading coefficient of $P_n$ is 
$$
A_n=\frac{4^n +5\cdot 2^n+10}{2^{4+(n-1)(2^n-2)}(2^n-2)!t_n}B_n.
$$

\subsection{Proof of Theorem~\ref{thm:GenSer}} 
Using  $f_n(d) =  F_n(d)- F_n(d-1)$ and~\eqref{FnX}, we write 
\begin{equation}
\label{eq:GnFn}
\begin{split} 
g_n(s)
= \frac{1}{t_n} &\sum_{W\in\{2^{n+1}, 3\cdot 2^{n-1}, 2^{n },0\}} C_n(W)\\
&  \sum_{d=1}^\infty \frac{1}{d^{s}} 
\(T_n\(\frac{d^{1/2^{n-1}}}{2^{W/2^{n-1}}}\)- T_n\(\frac{(d-1)^{1/2^{n-1}}}{2^{W/2^{n-1}}}\)\).
\end{split}
\end{equation}
Note that if there is an integer $m$ with 
$$
\frac{d^{1/2^{n-1}}}{2^{W/2^{n-1}}}\geqslant m > \frac{(d-1)^{1/2^{n-1}}}{2^{W/2^{n-1}}}
$$
then $d \geqslant  2^{W} m^{2^{n-1}} > d-1$.
Hence this is possible if and only if $d = 2^{W} m^{2^{n-1}}$.
We now see from~\eqref{eq:Tn mu} that 
\begin{align*} 
T_n\(\frac{d^{1/2^{n-1}}}{2^{W/2^{n-1}}}\) &- T_n\(\frac{(d-1)^{1/2^{n-1}}}{2^{W/2^{n-1}}}\)\\
& =  \begin{cases}
\mu^2(m) (2^n-1)^{\omega(m)}, & \text{if} \ d = 2^{W} m^{2^{n-1}} \ \text{with}\ m \in \N, \\\
0, & \text{otherwise}.
\end{cases}
\end{align*}
Substituting this in~\eqref{eq:GnFn}, we easily obtain 
$$
g_n(s)
=  \frac{1}{t_n}  \sum_{W\in\{2^{n+1}, 3\cdot 2^{n-1}, 2^{n },0\}} C_n(W)
 \sum_{m=1}^\infty \frac{1}{( 2^{W} m^{2^{n-1}})^{s}} \mu^2(m) (2^n-1)^{\omega(m)}
$$
and the result follows.  

\subsection{Proof of Theorem~\ref{thm:MaxGal}}

As, usual we say that an integer $a$ is $Q$-friable if all prime
divisors of $a$ do not exceed $Q$.  Let $\psi(H,Q)$ denote the number
of positive $Q$-friable integers up to $H$, and let
$$
u = \frac{\log H}{\log Q}
$$
By~\cite[Part~III, Theorem~5.13]{Ten} and Hildebrand's theorem~\cite{H86} for $H \geqslant Q > 2$ we have 
\begin{equation}
\label{eq:Smooth H}
\psi(H,Q)\ll H {u}^{-u }  \end{equation} 
for $\log Q\geqslant  (\log\log H)^{5/3+\varepsilon}$ and any fixed $\varepsilon>0$.

Furthermore, we recall the classical asymptotic formula
\begin{equation}
\label{eq:Sqf H}
\#\(\Sqf  \cap [1,H]\) = \frac{1}{\zeta(2)} H + O\(H^{1/2+o(1)}\).
\end{equation}
where as before $\Sqf$ is the set of square-free integers, 
see~\cite[Theorem~334]{HardyWright} (note that using
the currently best known result of Jia~\cite{Jia} with $17/54$ instead of the exponent 
$1/2$ does not affect our final result).

Finally, for $Q \leqslant H$, we have the trivial bound
\begin{equation}
\begin{split}
\label{eq:gcd Q}
\#\left\{  \vec{a} \in  \fB_n\(H\) :~
\pwgcd(\vec{a})>Q\right\} & \leqslant \frac{n(n-1)}{2}H^{n-2} \sum_{d >Q} \fl{H/d}^2 \\
& \ll H^n Q^{-1}, 
\end{split}
\end{equation}
where for  $\vec{a} = (a_1, \ldots, a_n) \in \N^n$ we define
the pair-wise greatest common divisor $\pwgcd(\vec{a})$ as 
$$ 
\pwgcd(\vec{a}) = \max_{1\leqslant i< j\leqslant n}  \gcd(a_i,   a_j).
$$
For a real $Q\geqslant 2$  we  define
\begin{align*}
\cT_n&(H,Q)\\
& =
\{ \vec{a} \in  \Sqf^n \cap \fB_n\(H\) :~
\pwgcd(\vec{a}) \leqslant Q \text{ and no $a_{i}$ is $Q$-friable} \}. 
\end{align*}

Combining~\eqref{eq:Smooth H}, \eqref{eq:Sqf H}   and~\eqref{eq:gcd Q}, we derive
\begin{equation}
\label{eq:T asymp}
\#\cT_n(H,Q)   = H^n\( \frac{1}{\zeta(2)^n}  + O\(H^{-1/2+o(1)} +   u^{-u} +    Q^{-1}\)\).
\end{equation}

We now claim that if $\vec{a}, \vec{b} \in \cT_n(H,Q)$ generate the same
multiquadratic field (with full Galois group), then they agree up to a
permutation of coordinates.

We see this as follows: applying the map $\varphi_{H}$, given
by~\eqref{eq:Map phi}, componentwise, we may regard $\vec{a}, \vec{b}$
as two $\ftwo$ matrices, with $n$ rows and $\pi(H)$ columns.
Moreover, by the nonfriability assumption on $\vec{a} \in \cT_n(H,Q)$
(together with the assumption of square-freeness), each $\varphi_{H}(a_i)$
has a one in some $p$-indexed column for some prime $p > Q$.

Moreover, for $p>Q$, using the condition on $\pwgcd(\vec{a})$, we note that there can
be at most one nonzero element in each column.  That is, each $a_{i}$
gives rise to some $p_{i} > Q$ such that the $p_{i}$-column has a one
in row $i$, and zeros elsewhere. Recalling Lemma~\ref{lem:Span}, 
this implies that for any  $\vec{a}\in \cT_n(H,Q)$ 
we have   $\gal\(\Q\(\sqrt{\vec{a}}\)/\Q\) \simeq (\Z/2\Z)^{n}$. 

Now, if the fields are the same, we must have ramification at the same
primes. In particular, we see from Lemma~\ref{lem: discr} that for
each $i=1, \ldots, n$ there must exist some $j_{i}$, $1 \leqslant j_i\leqslant n$,
such that $p_{i} \mid b_{j_{i}}$.  Thus, after permuting rows in the
matrix associated with $\vec{b}$, and using that the conditions
$\pwgcd(\vec{b}) \leqslant Q$, also holds for $\vec{b}$, we find that the
matrices associated to $\vec{a}$ and $\vec{b}$ are identical in the
columns indexed by $p_{1}, \ldots, p_{n}$; by permuting the rows of
the two matrices, both
restrictions to these columns are in fact the identity matrix.

Using that the fields $\Q\(\sqrt{\vec{a}}\)$ and $\Q\big(\sqrt{\vec{b}}\big)$ are the same if and only if the associated
$\ftwo$-vectors generated by the map $\varphi_{H}$ have the same span,
there must exist some 
matrix $M \in \GL_{n}(\ftwo)$ that maps the matrix associated with
$\vec{a}$ into the 
matrix associated with $\vec{b}$; comparing columns indexed by
$p_{1}, \ldots, p_{n}$ we find that $M$ is in fact the identity
matrix, provided that we have permuted the rows as above (note
that reordering the rows amounts to reordering the entries in
$\vec{a}, \vec{b}$.)

Thus, after permuting the rows  in $\vec{b}$ as described
above we find that $\vec{a}$ and $\vec{b}$ are the same.

Hence 
\begin{equation}
\label{eq:G T}
G_n(H)  \geqslant \frac{1}{n!} \#\cT_n(H,Q)  + O(H^{n-2}Q),
\end{equation}
where the error term comes from vectors $\vec{a}$ with two identical components
(which cannot exceed $Q$).

It is also obvious that  alternatively we can define
$G_n(H) $ using only vectors $\vec{a}$ with square-free
components, that is,  as
\begin{align*}
&G_n(H)\\
& \quad = \# \left\{\Q\(\sqrt{\vec{a}}\):~\vec{a}\in  \Sqf^n \cap \fB_n\(H\)\
\text{and}\ \# \gal\(\Q\(\sqrt{\vec{a}}\)/\Q\) = 2^n \right \}.
\end{align*}

Thus, recalling~\eqref{eq:Sqf H},  we immediately obtain 
\begin{equation}
\label{eq:G UP}
G_n(H)  \leqslant H^n\( \frac{1}{n!\zeta(2)^n}  + O\(H^{-1/2+o(1)}\)\).
\end{equation}
Combining~\eqref{eq:T asymp} and~\eqref{eq:G T} with~\eqref{eq:G UP}, we 
obtain
$$
G_n(H)  =  H^n\( \frac{1}{n!\zeta(2)^n}  +  O\(H^{-1/2+o(1)} +u^{-u } +    Q^{-1}+ H^{-2}Q \)\).
$$
Choosing 
\begin{equation}
\label{eq:Q}
Q =  \exp\(\sqrt{ ( \log H)(\log \log H)/2}\)
\end{equation}
so that $u =  \sqrt{2( \log H)/(\log \log H)}$, we 
conclude the proof.

\subsection{Proof of Theorem~\ref{thm:MaxDeg} }
\label{sec:proof-theorem-maxdeg}

First recall that $Z_{k} = \Q(\zeta_k)$ denotes the $k$-th cyclotomic
field.  We  use Kummer theory to analyze the extension
$Z_{k} K_{\ba} / Z_{k}$ and then use the fact that
$[Z_{k} K_{\ba} : Z_{k}] = k^{n}$ implies that $[K_{\ba} : \Q] = k^{n}$.
By Kummer theory, (cf.~\cite[Section~14.7]{DuFo} 
or~\cite[Chapter~VI, Sections~8--9]{lang-algebra}) we see that $\gal(Z_{k} K_{\ba} / Z_{k})$ is 
isomorphic to  
$$
(\langle a_{1}, \ldots, a_{n} \rangle  (Z_{k}^{\times})^{k}) / (Z_{k}^{\times})^{k},
$$
where $(Z_{k}^{\times})^{k}$
denotes the $k$-th powers in $Z_{k}^{\times}$.
We begin by showing that any relation, modulo $k$-th powers in
$Z_{k}^{\times}$, must already be a relation modulo $k$-th powers in
$\Q^{\times}$. 
\begin{lemma}
 If $k \geqslant  3$ is an odd integer then the map
 $$
 \Q^{\times}/(\Q^{\times})^{k} \to
 Z_{k}^{\times}/(Z_{k}^{\times})^{k}
 $$
 is injective. In particular, an element $\alpha \in \Q^\times$ is a
 $k$-th power in $Z_k$ if and only if $\alpha \in \Q^{\times k}$
\end{lemma}
\begin{proof}

 We first recall that
 $t^{k} - \alpha$ is irreducible over $\Q$ (cf.~\cite[Theorem~9.1,
 Chapter~VI, Section~9]{lang-algebra}) provided that $\alpha$ is not a
 $p$-th power of some rational number, for all prime divisors $p \mid k$.
 
  Now, let $\alpha$ denote a element in the kernel of the above map,
 and assume that $\alpha$ is not a $k$-th power of any element in
 $\Q$.  If $\alpha = \alpha_{1}^{p}$ for some $p|k$ and
 $\alpha_{1} \in \Q$, write $k = pr$ and note that
 $t^{pr}-\alpha_{1}^{p} = \prod_{i=1}^{p}( t^{r} - \zeta_{p}^{i}
 \alpha_{1})$.  Thus, if $t^{k}-\alpha_{1}^{p}$ has a root in
 $Z_k$, there exists $i$ such that
$t^{r} - \zeta_{p}^{i} \alpha_{1}$ has a root in
 $Z_k$ which, as $\zeta_{p}^{i}=\zeta_{k}^{ir}$, implies 
that $t^{r}- \alpha_{1}$ has a root in $Z_{k}$.
 Repeating this procedure a finite number of times, we may thus
 reduce to the case of showing that the irreducible polynomial
 $t^{r_{\ell}} - \alpha_{\ell}$ does not have any roots in $Z_{k}$, for
 $\alpha_{\ell} \in \Q \smallsetminus \{ \pm 1 \}$, and $\alpha_\ell$ not a
 $p$-th power for any prime $p \mid  r_{\ell} \mid  k$.  However, by~\cite[Theorem~9.4, Chapter~VI, Section~9]{lang-algebra}, 
 the Galois group of $t^{r_\ell} - \alpha_{\ell}$ is nonabelian, and hence the roots cannot
 be contained in $Z_{k}$ since the cyclotomic extension
 $Z_{k}/\Q$ is abelian.
\end{proof}

Thus, to count fields $K_{\vec{a}}$ with maximal degree is the same as
counting $\vec{a}=(a_{1},\ldots, a_{n})$ such that the group
$ < a_{1}, \ldots, a_{n} > (\Q^{\times})^{k}/ (\Q^{\times})^{k}$ has
cardinality $k^{n}$ --- in other words, counting tuples
$(a_{1}, \ldots, a_{k})$ such that $a_{1},\ldots, a_{k}$ are
independent modulo $k$-th powers in $\Q^\times$.

With $\Sqf_{k}$ denoting  the set of $k$-free integers, we have 
$$
\#\(\Sqf_{k}  \cap [1,H]\) = \frac{1}{\zeta(k)} H + O\(H^{1/k}\).
$$

As in the case of squares, 
we can define
$G_n^{k}(H) $ using only vectors $\vec{a}$ with $k$-free
components, that is,  as
$$ 
G_n^{k}(H)
= \# \left\{\Q\(\sqrt{\vec{a}}\):~\vec{a}\in  \Sqf_{k}^n \cap \fB_n\(H\)\
\text{and}\  [\Q\(\sqrt[k]{\vec{a}}\):\Q] = k^n
               \right \}.
$$

Restricting to the set of ``nice'' $\ba$ as in the argument for
multi-quadratic fields (that is, to the set of vectors $\ba$ having no $Q$-friable
component $a_{i}$, as well making sure any pairwise greatest common divisor is at most
$Q$), the argument is essentially the same except for one small
caveat: if $k$ is not prime, we cannot use linear algebra over a
finite field, but must rather work with the finite ring
$\Z/k\Z$. However, as $\End( (\Z/k\Z)^{n} ) \simeq
\Mat_{n}( \End(\Z/k\Z))$ and the set of
invertible endomorphisms can  be identified with $\GL_{n}(\Z/k\Z)$ the
previous argument applies also for $k$ not prime.

Choosing $Q$ as in~\eqref{eq:Q}, we 
conclude the proof. 

%


\section*{Acknowledgment}

The authors would like to  thank Kevin Destagnol  for many useful comments on the initial version of this paper and Jean-Fran\c cois Mestre for explaining  the proof of Lemma~\ref{lem: discr}. 

The authors are also very grateful to the referee for the very careful reading of the manuscript and many 
valuable suggestions.  

This work started during a very enjoyable visit by I.S. to the 
Institute of Mathematics of Jussieu (Paris) and to the Department of Mathematics 
of  KTH (Stockholm), whose support and hospitality are gratefully acknowledged. 

During this work,
P.K. was partially supported by the 
Swedish Research Council (2016-03701) and  I.S. was supported  by  the
Australian Research Council (DP170100786).

\end{document}